\documentclass[12pt]{amsart}
\usepackage{times}
\usepackage{mathptmx}
\usepackage{amssymb,amsmath,amsthm,newlfont,enumerate}
\usepackage{enumitem}

\usepackage[margin=3cm]{geometry}
\usepackage{xcolor}
\usepackage{multicol}
\usepackage[colorlinks = true, linkcolor = blue]{hyperref}

\theoremstyle{plain}
\newtheorem{theorem}{Theorem}

\newtheorem{lemma}[theorem]{Lemma}
\newtheorem{corollary}[theorem]{Corollary}

\theoremstyle{definition}
\newtheorem{definition}[theorem]{Definition}

\numberwithin{equation}{section}

\usepackage{multicol}
\usepackage[colorlinks = true, linkcolor = blue]{hyperref}

\newcommand{\im}[1]{\mathbf{ i_{#1}}}
\newcommand{\ims}[1]{\mathbf{ i_{#1}^2}}
\newcommand{\jm}[1]{\mathbf{ j_{#1}}}
\newcommand{\jms}[1]{\mathbf{ j_{#1}^2}}

\newcommand{\be}[1]{\mathbf{e_{#1}}}
\newcommand{\bes}[1]{\mathbf{e_{#1}^2}}

\newcommand{\mB}{\mathbb{B}}
\newcommand{\mC}{\mathbb{C}}
\newcommand{\mD}{\mathbb{D}}

\newcommand{\mR}{\mathbb{R}}

\newcommand{\ra}{\rightarrow}

\renewcommand{\Re}{\operatorname{Re}}
\renewcommand{\Im}{\operatorname{Im}}

\begin{document}

%\articletype{ARTICLE TEMPLATE}% Specify the article type or omit as appropriate

\title{Involutions of Bicomplex Numbers}
\author{Pierre-Olivier Paris\'e}
\thanks{\textit{Corresponding Author Email.} parisepo@hawaii.edu}
\address{Pierre-Olivier Parisé: Department of Mathematics, University of Hawai'i at Manoa,
Honolulu, Hawai'i,  United-States, 96822}
\email{parisepo@math.hawaii.edu}

\begin{abstract}
An involution of a real commutative algebra $A$ is a real-linear homomorphism $f : A \rightarrow A$ such that $f^2 = \mathrm{Id}$. We show that there are six involutions of the algebra of bicomplex numbers, contrary to the actual number of four stated in the literature. We also characterize $n$-involutions satisfying the additional property $f^n = \mathrm{Id}$ for some integer $n \geq 2$. We show there are eight $n$-involutions and they occur only for $n = 2$ and $n= 4$. We use our result to give a new characterization  of the invertible elements of the algebra of bicomplex numbers.  
\end{abstract}



\maketitle
Word Counting: 3,994.

\section{Introduction}
Let $\mC$ be the set of complex numbers together with its usual addition and multiplication making it a commutative field. It is an easy exercise to show that the only real-linear homomorphisms $f : \mC \ra \mC$ such that $f (f(z)) = z$ for any complex number $z$ are $f(z) = z$ and $f(z) = \overline{z}$, the complex conjugate. The functions satisfying these properties (real-linear homomorphism and $f\circ f = \mathrm{Id}$) will be called \textit{involutions}. In fact, for an algebra $A$ over a commutative ring $R$, a function $f : A \ra A$ is said to be an involution if $f^2 (a) = a$, that is $f$ is its own inverse, and $f$ is a real-linear homomorphism of $A$. We may ask therefore if there are other characterizations of this type for other algebras of numbers.

In \cite{SangwineTodd2007}, the authors were interested in the (anti)involutions of the algebra of quaternions and its connections to the description of rotations and reflections in $\mR^3$. 
The quaternion algebra were introduced by Hamilton in the 19th century. They are defined as the set $\mathbb{H}$ of expressions $q = a + b\im{} + c \jm{} + d \mathbf{k}$ where $\ims{} = \jms{} = \mathbf{k^2} = -1$ and $a, b, c, d \in \mR$, together with an addition defined component wise and a multiplication of the imaginary units $\im{}$, $\jm{}$, and $\mathbf{k}$ defined by the following rules: 
	$$
	\im{} \jm{} = -\jm{} \im{} = \mathbf{k}, \quad \jm{} \mathbf{k} = -\mathbf{k} \jm{} = \im{}, \quad \mathbf{k} \im{} = -\im{} \mathbf{k} = \jm{}.
	$$ 
With these operations, the quaternions become a noncommutative field (or a noncommutative algebra over $\mR$). A pure quaternion is a quaternion ${\boldsymbol\mu} = a \im{} + b \jm{} + c \mathbf{k}$. The authors showed that the functions of the form $q \mapsto - {\boldsymbol\mu} q {\boldsymbol\mu}$ for a given pure quaternion ${\boldsymbol\mu} = a \im{} + b \jm{} + c \mathbf{k}$, with $a^2 + b^2 + c^2 = 1$, are involutions of the quaternions $\mathbb{H}$. In \cite{Lawon2021}, the authors then show that this is in fact the only involutions of the quaternions, completing the characterization of the involutions of the quaternions. The  results of the authors of \cite{SangwineTodd2007} were also generalized to dual quaternions and dual split quaternions in \cite{Bek2013, Bek2016b}.

In this paper, we are interested in the same problem of characterizing involutions but in the algebra of bicomplex numbers, denoted by $\mB \mC$. The precise definition of these numbers is presented in Section \ref{sec:AlgBC}. The main characteristic of the bicomplex numbers is that, endowed with an addition and a multiplication, they become a commutative ring with zero divisors. The presence of zero divisors makes the algebra of bicomplex numbers rather different from the algebra of quaternions. 

According to the literature on the bicomplex numbers, for example \cite{Bicomplex, RochonShapiro} where the involution were first mentionned, only three possible non-trivial involutions on the set of bicomplex numbers are presented: if $s = z_1 + z_2\im{2}$ is a bicomplex number, where $z_1, z_2$ are complex numbers, then
	\begin{align*}
	s^{\star} = \overline{z}_1 + \overline{z}_2 \im{2} , \quad s^{\ast} = z_1 - z_2 \im{2} , \quad s^{\dagger} = \overline{z}_1 - \overline{z}_2 \im{2} .
	\end{align*}	
To the best of our knowledge, no attempt was made to show that these are in fact the only involutions of the bicomplex numbers. We however show that there are more involutions than expected. Our main result is the following.
	\begin{theorem}\label{T:MainOne}
	Let $s = z_1 + z_2 \im{2}$ be a bicomplex number. Then there are six possible involutions:
		\begin{enumerate}
		\item $(z_1 + z_2 \im{2})^{\dagger_0} = z_1 + z_2 \im{2}$.
		\item $(z_1 + z_2 \im{2})^{\dagger_1} = z_1 - z_2 \im{2}$.
		\item $(z_1 + z_2 \im{2})^{\dagger_2} = \overline{z}_1 + \overline{z}_2 \im{2}$;
		\item $(z_1 + z_2 \im{2})^{\dagger_3} = \overline{z}_1 - \overline{z}_2 \im{2}$;
		\item $(z_1 + z_2 \im{2})^{\dagger_4} = (\Re (z_1) - \Re (z_2) \im{1}) + (-\Im (z_1) + \Im (z_2) \im{1}) \im{2}$;
		\item $(z_1 + z_2 \im{2})^{\dagger_5} = (\Re (z_1) + \Re (z_2) \im{1}) + (\Im (z_1 ) + \Im (z_2) \im{1}) \im{2}$.
		\end{enumerate}
	\end{theorem}
The involutions $\dagger_1$, $\dagger_2$, and $\dagger_3$ corresponds to the known involutions $\ast$, $\star$, and $\dagger$ respectively. The last two were unknown yet in the literature. We also characterize all real-linear homomorphisms $f : \mB \mC \ra \mB \mC$ such that $f^n = \mathrm{Id}$ where $\mathrm{Id}$ is the identity mapping and $n \geq 2$ is a positive integer. These functions are called $n$-involutions and our main result shows there are eight $n$-involutions and they only occur for $n = 2$ and $n = 4$. Our main results give another important difference between the algebra of bicomplex numbers and the algebra of quaternions: in the former, there are finitely many involutions and, in the latter, there are infinitely many involutions! The proofs of our main results are presented in Section \ref{sec:CharInv}. 

In Section \ref{sec:GroupStructure}, we give a description of the group structures of the $n$-involutions. We use the $n$-involutions to give a new characterization of invertible bicomplex numbers. Our description was based on a result stated in \cite{RochonValiere}. Finally, in the last section, Section \ref{sec:GeomInterpretation}, we provide some geometric interpretations of the $n$-involutions of bicomplex numbers.

\section{Algebra of Bicomplex numbers}\label{sec:AlgBC}
The bicomplex numbers, or more generally, the multicomplex numbers, were introduced by Segre \cite{Segre} and Cockle \cite{cockle1, cockle2, cockle3, cockle4} to give another generalization of the complex numbers apart from the quaternions. For a modern treatment of the bicomplex numbers, we refer the reader to \cite{Bicomplex} or \cite{Baley}. Here, we follow the presentation from \cite{Baley}.

A bicomplex number is defined by a duplication process of the complex numbers. To fix the notation, we will use the imaginary unit $\im{1}$ such that $\ims{1} = -1$ and produce the set
	\begin{align*}
	\mC (\im{1}) := \{ x + y \im{1} \, : \, x, y \in \mR \}
	\end{align*}
of complex numbers. The set of bicomplex numbers is then defined in the following way.

	\begin{definition}
	The set of bicomplex numbers, denoted by $\mB \mC$, is the set
		\begin{align*}
		\mB \mC := \{ z_1 + z_2 \im{2} \, : \, z_1 , z_2 \in \mC (\im{1}) \}
		\end{align*}
	with $\ims{2} = -1$ and $\im{2} \neq \im{1}$.
	\end{definition}
	\noindent The set of bicomplex numbers can also be viewed as a two-dimensional complex Clifford algebra (see \cite{RochonShapiro}).

If we endow the set $\mB \mC$ with the following operations: if $s = z_1 + z_2 \im{2}$ and $t = w_1 + w_2 \im{2}$, then
	\begin{itemize}
	\item $s + t = (z_1 + w_1) + (z_2 + w_2) \im{2}$;
	\item $s \cdot t = (z_1 w_1 - z_2 w_2) + (z_1 w_2 + z_2 w_1) \im{2}$.
	\end{itemize}
then the triplet $(\mB \mC , + , \cdot )$ is a commutative ring with zero divisors. Two divisors of zero play an important role in the theory of bicomplex numbers. We shall talk about this fact later in this section.

	If $s$ is a bicomplex number, then by writing down explicitly the real and imaginary parts of the components of $s$, we then obtain the following representation in terms of real components:
	\begin{align}
	s & = x_1 + x_{\im{1}} \im{1} + x_{\im{2}} \im{2} + x_{\jm{1}} \jm{1} 
	\end{align}
where $\jm{1} = \im{1}\im{2}$ is called a hyperbolic unit because $\jms{1} = 1$. Using a hyperbolic unit, we can define the set of hyperbolic numbers:
	\begin{align*}
	\mD (\jm{1} ) := \{ x + y \jm{1} \, : \, x , y \in \mR \} .
	\end{align*}	
Equipped with the bicomplex multiplication, the set of hyperbolic numbers becomes a commutative subring of the set of bicomplex numbers. The hyperbolic numbers have a strong connection with the hyperbolic cosine and hyperbolic sine functions similar to the connection the complex numbers have with the cosine and sine functions. For more details on this subject, the reader is referred to Sobczyk's article \cite{Sobczyk}.

We mentioned that the bicomplex numbers form a commutative ring with non invertible elements. This is one of the major differences with the quaternions. Fortunately, in the case of the bicomplex numbers, we can completely characterize the non-invertible elements. To present the characterization, we introduce the following bicomplex numbers:
	\begin{align*}
	\be{1} := \frac{1 + \jm{1}}{2}  \quad \text{and} \quad \be{2} := \frac{1 - \jm{1}}{2} .
	\end{align*}
These two bicomplex numbers are called idempotent elements because of the following properties:
	\begin{align*}
	\bes{1} = \be{1} \quad \text{ and } \quad \bes{2} = \be{2}. 
	\end{align*}
They also satisfy the following additionnal properties
	\begin{align*}
	\be{1} + \be{2} = 1 \quad \text{ and } \quad \be{1} \be{2} = 0 .
	\end{align*}
Multiplying a bicomplex number $s$ respectively by $\be{1}$ and $\be{2}$, we obtain
	\begin{align*}
	s \be{1} = (z_1 - z_2 \im{1}) \be{1} \quad \text{ and } \quad s \be{2} = (z_1 + z_2 \im{1}) \be{2} .
	\end{align*}
Adding these last two identities and using the property $\be{1} + \be{2} = 1$, we then see that a bicomplex number $s$ can be rewritten in the following form
	\begin{align}
	s = (z_1 - z_2 \im{1} ) \be{1} + (z_1 + z_2 \im{1} ) \be{2} . \label{Eq:IdempRepr}
	\end{align}
This is called the idempotent representation of a bicomplex number. The complex numbers $z_1 - z_2\im{1}$ and $z_1 + z_2 \im{1}$ are called the idempotent components. To simplify the notation, we denote the idempotent components by $z_{\be{1}}$ and $z_{\be{2}}$ respectively. This implies that the idempotent representation of $s$ can be rewritten in the following way:
	\begin{align*}
	s = z_{\be{1}} \be{1} + z_{\be{2}} \be{2} .
	\end{align*}

The idempotent representation is useful because of the following properties: given two bicomplex numbers $s = z_{\be{1}} \be{1} + z_{\be{2}} \be{2}$ and $t = w_{\be{1}} \be{1} + w_{\be{2}} \be{2}$, we have
	\begin{itemize}
	\item $s = 0$ if and only if $z_{\be{1}} = z_{\be{2}} = 0$;
	\item $s = t$ if and only if $z_{\be{1}} = w_{\be{1}}$ and $z_{\be{2}} = w_{\be{2}}$;
	\item $s + t = (z_{\be{1}} + w_{\be{1}}) \be{1} + (z_{\be{2}} + w_{\be{2}}) \be{2}$;
	\item $st = (z_{\be{1}} w_{\be{1}}) \be{1} + (z_{\be{2}} w_{\be{2}}) \be{2}$.
	\end{itemize}
So the bicomplex multiplication can be performed by simply multiplying the idempotent components together. This is a very powerful tool that can be used to extend many results from Complex Analysis to the bicomplex setting.

We can now use the idempotent representation to describe all non-invertible elements (see \cite[Corollary 6.7]{Baley}).
	\begin{theorem}\label{T:InvertibilityBC}
	A bicomplex number $s = z_{\be{1}} \be{1} + z_{\be{2}} \be{2}$ is not invertible if and only if $z_{\be{1}} = 0$ or $z_{\be{2}} = 0$.
	\end{theorem}
	In Section \ref{sec:GroupStructure}, we will give another characterization of invertible elements using our result on bicomplex involutions.
	
\section{Characterization of Involutions}\label{sec:CharInv}
Let recall more precisely the definition of an involution on the set of bicomplex numbers.
	\begin{definition}\label{D:BCInvolution}
	A function $f : \mB \mC \ra \mB \mC$ is an involution if it satisfies the following properties:
		\begin{enumerate}
		\item $f (f(s )) = s$ for any $s \in \mB \mC$.
		\item $f (s + t ) = f(s) + f(t)$ and $f (\lambda s ) = \lambda f(s)$ for any $s, t \in \mB \mC$ and $\lambda \in \mR$.
		\item $f (st ) = f(s) f(t)$ for any $s, t \in \mB \mC$.
		\end{enumerate}
	\end{definition}
	The usual definition of an involution involves only the first condition. However, for the quaternions and general algebra over  commutative field (see \cite{SangwineTodd2007} and \cite{Lawon2021}), the above definition was adopted. Therefore, to compare our result with the one obtained by the authors of \cite{SangwineTodd2007}, we will stick to the above definition. If we take a closer look at our definition of the term ``involution'', we require that the function is a real linear homomorphism which is its own inverse. 
	
	We will now show our main result presented in the Introduction, that is Theorem \ref{T:MainOne}. %Recall that the first four involutions in the statement were already discovered in the paper \cite{RochonShapiro}. Our result is new because of the discovery of the last two involutions. 
	Before doing the proof, we will first prove some lemmas.
	
		\begin{lemma}\label{L:SquareMinusOne}
		The only bicomplex numbers squaring to $-1$ are $\im{1}, -\im{1} , \im{2} , -\im{2}$.
		\end{lemma}
		\begin{proof}
		Let $s = z_{\be{1}} \be{1} + z_{\be{2}} \be{2}$ be a bicomplex number written in its idempotent representation. Then we have
			\begin{align*}
			s^2 = z_{\be{1}}^2 \be{1} + z_{\be{2}}^2 \be{2} .
			\end{align*}
		Since $-1 = (-1) \be{1} + (-1) \be{2}$, we obtain
			\begin{align*}
			z_{\be{1}}^2 = -1 \quad \text{ and } \quad z_{\be{2}}^2 = -1 .
			\end{align*}
		Since $z_{\be{1}}$ and $z_{\be{2}}$ are complex numbers, we therefore conclude that the only possibilities for $z_{\be{1}}$ and $z_{\be{2}}$ are $\im{1}$ or $-\im{1}$. Replacing in the idempotent representation of $s$, we obtain the four possible cases in the statement of the lemma.
		\end{proof}
	If we compare this Lemma \ref{L:SquareMinusOne} with what we know on the quaternions, then we find a second major difference between the two algebras. Indeed, any pure quaternion $q = a \im{} + b \jm{} + c \mathbf{k}$ with $a^2 + b^2 + c^2 = 1$ squares to $-1$. In other words, there are infinitely many quaternions $q$ such that $q^2 = -1$. %In a sense, Lemma \ref{L:SquareMinusOne} tells us that the bicomplex numbers looks more like the complex numbers than the quaternions do.
	
	The next lemma combined with the previous one explain why we encounter only finitely many involutions of the bicomplex numbers.
		\begin{lemma}\label{L:ActionImag}
		If $f : \mB\mC \ra \mB \mC$ is a bijection satisfying the conditions (2) and (3) in Definition \ref{D:BCInvolution}, then its action is uniquely determined by its action on the imaginary units $\im{1}$ and $\im{2}$.
		\end{lemma}
		\begin{proof}
		Since $f$ is real linear, given any bicomplex number $s = x_1 + x_{\im{1}} \im{1} + x_{\im{2}} \im{2} + x_{\jm{1}} \jm{1}$, we can decompose the action of $f$ in the following way:
			\begin{align*}
			f (s) = x_1 f(1) + x_{\im{1}} f(\im{1}) + x_{\im{2}} f (\im{2}) + x_{\jm{1}} f(\jm{1} ) .
			\end{align*}
		We therefore see that the action of $f$ is determined by its action on the units $1$, $\im{1}$, $\im{2}$ and $\jm{1}$. However, we know that $\jm{1} = \im{1} \im{2}$ and therefore, since $f$ is a homomorphism, $f (\jm{1}) = f (\im{1}) f (\im{2})$. Also, since $1$ is the unit for the bicomplex multiplication, we must have $f(1) = 1$. This concludes the proof.
		\end{proof}
	
	We will now need two more lemmas. In the first lemma, by a signed imaginary unit $\im{}$, we mean any element of the set $\{ \im{1} , -\im{1} , \im{2} , -\im{2} \}$. Also, by a signed hyperbolic unit, we mean any element of the set $\{ \jm{1} , -\jm{1} \}$. 
		\begin{lemma}\label{L:ImagToImagHypToHyp}
		If $f$ is a bijection that satisfies the conditions (2) and (3) in Definition \ref{D:BCInvolution}, then $f$ maps any signed imaginary unit to a signed imaginary unit and any signed hyperbolic unit to a signed hyperbolic unit.
		\end{lemma}
		\begin{proof}
		We already know that $f(1) = 1$. Since $f$ is bijective, this implies that $f(\im{}) \neq 1$ and $f (\jm{}) \neq 1$ for any imaginary unit $\im{}$ and hyperbolic unit $\jm{}$. Suppose that some signed imaginary unit $\im{}$ is mapped to a signed hyperbolic unit $\jm{}$, that is $f(\im{}) = \jm{}$. Since $\jms{} = 1$, we therefore have
			\begin{align*}
			f(\im{})^2 = \jms{} = 1
			\end{align*}
		However, we have $f(\im{})^2 = f(\ims{}) = -1$ since $\ims{} = -1$. This is a contradiction and we must conclude that $f$ maps each signed imaginary unit on a signed imaginary unit and a signed hyperbolic unit on a signed hyperbolic unit.
		\end{proof}
	
		\begin{lemma}\label{L:NumberBijections}
		There are only $8$ bijections of $\mB \mC$ satisfying the conditions (2) and (3) in Definition \ref{D:BCInvolution}.
		\end{lemma}
		\begin{proof}
		By Lemma \ref{L:ActionImag}, we only need to figure out the possible values of the involution $f$ on the imaginary units $\im{1}$ and $\im{2}$. 
	
	Since $f$ is bijective, if $f (\im{1}) = \pm\im{1}$, then by Lemma \ref{L:ImagToImagHypToHyp} we must have $f(\im{2}) = \pm\im{2}$ respectively. It may also happen that the imaginary units $\im{1}$ and $\im{2}$ are interchanged. So we separate in two cases: 
		\begin{itemize}
		\item When $\im{1}$ and $\im{2}$ are not interchanged: There are two possible values for $\im{1}$ and $\im{2}$ which are respectively $\pm \im{1}$ and $\pm\im{2}$. We then obtain four bijections.
		\item When $\im{1}$ and $\im{2}$ are interchanged: There are two possible values for $\im{1}$ and $\im{2}$, which are respectively $\pm \im{2}$ and $\pm \im{1}$. We then obtain four other bijections.
		\end{itemize}
	Combining all the cases altogether, we obtain eight bijections satisfying conditions (2) and (3) in Definition \ref{D:BCInvolution}.
		\end{proof}
		
	The proof of our main theorem is then a direct consequence of the last lemma.
	\begin{proof}[Proof of Theorem \ref{T:MainOne}]
	Since $f^2 = \mathrm{Id}$ where we recall that $\mathrm{Id}$ is the identity mapping, then $f$ is bijective. We know from the last lemma that there are only $8$ possible bijections satisfying the conditions (2) and (3) of Definition \ref{D:BCInvolution}. We will inspect each cases according to two categories:
		\begin{itemize}
		\item When $\im{1}$ and $\im{2}$ are not interchanged. In this case, there are two possible values for $\im{1}$ and $\im{2}$ which are respectively $\pm \im{1}$ and $\pm\im{2}$. Writing down explicitly each case, we obtain the expressions (1), (2), (3) and (4) in the list of Theorem \ref{T:MainOne}. In each case, the condition $f(f(s)) = s$ is satisfied.
		\item When $\im{1}$ and $\im{2}$ are interchanged. In this case, since $f (f (s)) = s$, the choice of sign for the value of $f(\im{1})$ will automatically determine the sign of the value of $f(\im{2})$. Then there are only two cases: $f(\im{1}) = \im{2}$ and $f (\im{2}) = \im{1}$ or $f (\im{1}) = -\im{2}$ and $f (\im{2}) = -\im{1}$. Writing down explicitly each case, we obtain the expressions (5) and (6) in the list of Theorem \ref{T:MainOne}. 
		\end{itemize}
	This completes the proof.
	\end{proof}
	
It is quite useful to have the expression of the involution in terms of the idempotent components. It is not difficult to accomplish this task with some basic manipulations.
	\begin{theorem}
	If $s = z_{\be{1}} \be{1} + z_{\be{2}} \be{2}$, then the idempotent expressions of each involution are respectively
		\begin{enumerate}
		\begin{multicols}{2}
		\item $s^{\dagger_0} = z_{\be{1}} \be{1} + z_{\be{2}} \be{2}$;
		\item $s^{\dagger_1} = z_{\be{2}} \be{1} + z_{\be{1}} \be{2} $;
		\item $s^{\dagger_2} = \overline{z}_{\be{2}} \be{1} + \overline{z}_{\be{1}} \be{2}$;
		\item $s^{\dagger_3} = \overline{z}_{\be{1}} \be{1} + \overline{z}_{\be{2}} \be{2}$;
		\item $s^{\dagger_4} = \overline{z}_{\be{1}} \be{1} + z_{\be{2}} \be{2}$.
		\item $s^{\dagger_5} = z_{\be{1}} \be{1} + \overline{z}_{\be{2}} \be{2}$;
		\end{multicols}
		\end{enumerate}
	\end{theorem}
	
A natural question to ask now would be: how many functions $f : \mB \mC \ra \mB \mC$ are there satisfying the conditions (2) and (3) of Definition \ref{D:BCInvolution}, together with the more general condition $f^n (s ) = s$ for any bicomplex number $s$, where $n\geq 2$ is an integer? We call the functions satisfying the previous conditions $n$-involutions where $n \geq 2$ is an integer. Note that $n$-involutions are still bijections because the condition $f^n = \mathrm{Id}$ means that $f^{n-1}$ is the inverse of $f$.

To answer completely this question, we need several lemmas. We will denote by $S_4$ the set of all permutations of $\{ \im{1} , -\im{1} , \im{2} , -\im{2} \}$. We use the same notation as the symmetric group because we are permutating four different symbols. For the symmetric group $S_4$, we will use the same term ``$n$-involution'' to refer to an element $\sigma \in S_4$ such that $\sigma^n = \sigma$ where $\sigma^n := \sigma \circ \sigma \cdots \circ \sigma$, the $n$-fold composition of $\sigma$.
	\begin{lemma}\label{L:InvolutionSubsetSFour}
	The set of $n$-involutions on $\mB \mC$ is a subset of the set of $n$-involutions of $S_4$.
	\end{lemma}
	\begin{proof}
	From Lemma \ref{L:ActionImag}, we know that the action of $f$ is completely determined by its action on the imaginary units $\im{1}$, $\im{2}$. The possible values on each unit are $\im{1} , -\im{1} , \im{2} , -\im{2}$. Then, the permutation $\sigma_f := (f (\im{1}) , f (-\im{1}) , f (\im{2}) , f (-\im{2}))$ is well-defined and the application $f \mapsto \sigma_f$ is an injection in $S_4$. It remains to show that the permutation $\sigma_f$ is an $n$-involution for $S_4$. If we apply $n$-times the permutation and if we use the fact that $f$ satisfies $f^n (s) = s$ for any bicomplex number $s$, we obtain 
		\begin{align*}
		\sigma_f^n = (f^n (\im{1}) , f^n (-\im{1}) , f^n (\im{2}) , f^n (-\im{2}) ) = (\im{1} , -\im{1} , \im{2} , -\im{2} ) .
		\end{align*}
	This confirms that the permutation is an $n$-involution for $S_4$ and it completes the proof.
	\end{proof}
Lemma \ref{L:InvolutionSubsetSFour} has the following consequence. In the sequel, a non-trivial $n$-involution for $S_4$ is an $n$-involution that is different from the identity permutation.
	\begin{corollary}\label{C:PrimeInvolution}
	The only $p$-involution of $\mB \mC$, for a prime number $p \geq 3$, is the identity.
	\end{corollary}
	\begin{proof}
	According to Lemma \ref{L:InvolutionSubsetSFour}, the set of $p$-involutions of $\mB \mC$, for $p$ a prime number, is a subset of the set of $p$-involutions of $S_4$. When $p > 3$, there are no non-trivial $p$-involutions in the symmetric group $S_4$ because $p > 4$. When $p = 3$, the possible $3$-involutions of $S_4$ are the identity plus the following eight permutations:
		\begin{itemize}
		\begin{multicols}{2}
		\item $(\im{1} , \im{2} , -\im{2}, -\im{1})$;
		\item $(\im{1} , -\im{2} , -\im{1} , \im{2})$;
		\item $(-\im{1}, \im{2} , \im{1} , -\im{2})$;
		\item $(-\im{1} , -\im{2} , \im{2} , \im{1})$;
		\item $(\im{2} , \im{1} , -\im{1} , -\im{2})$;
		\item $(\im{2} , -\im{1} , -\im{2} , \im{1})$;
		\item $(-\im{2} , \im{1} , \im{2} , -\im{1})$;
		\item $(-\im{2} , -\im{1} , \im{1} , \im{2})$. 
		\end{multicols}
		\end{itemize}
	The entries of the permutations correspond to the value of $f (\im{1})$, $f(-\im{1})$, $f (\im{2})$, and $f (-\im{2})$, respectively. A straightforward verification shows that none of the eight permutations in the list correspond to a valid $3$-involution on the algebra of bicomplex numbers. This comes from the basic fact that at least three symbols must be choosen to be interchanged and with this choice, the values attributed to $f$ on the set $\{ \im{1} , -\im{1} , \im{2} , -\im{2}\}$ will not make $f$ a bijection or not even a function.
	\end{proof}
	\begin{theorem}
	For any integer $n \geq 2$ different than a power of $2$, there is no non-trivial $n$-involution for $\mB \mC$.
	\end{theorem}
	Similar to the context of $S_4$, the term ``non-trivial $n$-involution'' for the set $\mB \mC$ refers to an $n$-involution for $\mB \mC$ which is different from the identity map.
	\begin{proof}
	From Theorem \ref{T:MainOne}, there are involutions ($2$-involutions) on the set $\mB \mC$. So, this implies that there are $n$-involutions when $n$ is any power of $2$. 
	
	From Corollary \ref{C:PrimeInvolution}, for every prime number $p \geq 3$, there are no non-trivial $p$-involution defined on $\mB \mC$. As a consequence of the Fundamental Theorem of Arithmetic or the Prime Factorization Theorem, we conclude that this is also the case for any other integer $n$ not equal to a power of $2$. This completes the proof.
	\end{proof}
	
	What we need to do now is to find all involutions of order $2^m$ where $m \geq 2$ is an integer, namely the $2^m$-involutions. From Lemma \ref{L:NumberBijections}, however, we simply have to check if the two cases excluded in the proof of Theorem \ref{T:MainOne} are $4$-involutions.
	\begin{corollary}\label{C:4InvolutionsBC}
	There are eight $4$-involutions of $\mB \mC$.
	\end{corollary}
	\begin{proof}
	This is a consequence of Lemma \ref{L:NumberBijections}. We already know that the first six bijections give rise to the involutions in Theorem \ref{T:MainOne}. So they are also $4$-involutions. It is straigthforward to show that the two remaining cases are $4$-involutions.
	\end{proof}
	Here are the expressions of the two remaining $4$-involutions:
		\begin{itemize}
		\item $s^{\ddagger_6} = \Re (z_1) - \Re (z_2) \im{1} + \Im (z_1) \im{2} - \Im (z_2) \jm{1} = \overline{z}_{\be{2}} \be{1} + z_{\be{1}} \be{2}$.
		\item $s^{\ddagger_7} = \Re (z_1) + \Re (z_2) \im{1} - \Im (z_1)\im{2} - \Im (z_2) \jm{1} = z_{\be{2}} \be{1} + \overline{z}_{\be{1}} \be{2}$.
		\end{itemize}
	
\section{Group Structure}\label{sec:GroupStructure}
We will adopt the following terminology from now on. Any $2$-involution on $\mB \mC$ is called a conjugate and any $4$-involution which is not a $2$-involution is called a pseudo-conjugate. The reason for this terminology is the following. The $2$-involutions satisfy the same properties as the complex conjugate and it is natural to call them conjugates. To be more precise, they are bicomplex conjugates. The $4$-involutions which are not $2$-involutions do not quite satisfy the property of the complex conjugate; they do not satisfy the condition $f (f (s )) = s$. However, we just need to apply them two other times to obtain the identity map. For this reason, they almost satisfy the property of being a conjugate.

From the work of Rochon and Shapiro \cite{RochonShapiro}, the four conjugates $\dagger_0$, $\dagger_1$, $\dagger_2$, and $\dagger_3$ form a group under the operation of composition of functions. The Cayley table of the relations between each conjugates is presented in Table \ref{Tb:CayleyTabFirstFourConj} (see the Appendix).	
	
	However, if we try to join $\dagger_4$ and $\dagger_5$ to the group, then we no longer have a group. A way to see it is to compose one of the conjugates, $\dagger_4$ or $\dagger_5$, with the first four conjugates and see that we obtain an element outside of the set of $\{ \dagger_0 , \dagger_1 , \dagger_2 , \dagger_3 , \dagger_4 , \dagger_5 \}$. We can compose $\dagger_4$ and $\dagger_1$. We obtain
		\begin{align*}
		(s^{\dagger_4})^{\dagger_1} = (z_{\be{1}} \be{1} + \overline{z}_{\be{2}} \be{2})^{\dagger_1} = \overline{z}_{\be{2}} \be{1} + z_{\be{1}} \be{2} = s^{\ddagger_6}.
		\end{align*}
	We know that $\ddagger_6$ is not a conjugate from the proof of Corollary  \ref{C:4InvolutionsBC}. We therefore conclude that the set $\{ \dagger_0 , \dagger_1 , \dagger_2 , \dagger_3 , \dagger_4 , \dagger_5 \}$ is not closed under the composition and is therefore not a group.
	
	If we change the point of view and consider the set of all $4$-involutions, we then obtain a group under the composition which is of order $8$. Simple calculations show that the Cayley table for the set $\{ \dagger_0 , \dagger_1 , \dagger_2 , \dagger_3 , \dagger_4 , \dagger_5 , \ddagger_6 , \ddagger_7 \}$ equiped with the composition is the Table \ref{Tb:CayleyEigthFunctions} (see the Appendix). 
	From this table, we see that we obtain a group under the composition and the order of the group is $8$ which is a power of $2$. We denote the group of conjugates and pseudoconjugates by ${\boldsymbol \dagger}$. The following summary of the six conjugations and the two $4$-involutions helped to construct the Cayley table:
		\begin{multicols}{2}
		\begin{enumerate}
		\item[$\dagger_0$:] $f(s) = s$ (identity).
		\item[$\dagger_1$:] $f(s) = z_{\be{2}} \be{1} + z_{\be{1}} \be{2}$.
		\item[$\dagger_2$:] $f(s) = \overline{z}_{\be{2}} \be{1} + \overline{z}_{\be{1}} \be{2}$.
		\item[$\dagger_3$:] $f(s) = \overline{z}_{\be{1}} \be{1} + \overline{z}_{\be{2}} \be{2}$.
		\item[$\dagger_4$:] $f(s) = z_{\be{1}} \be{1} + \overline{z}_{\be{2}} \be{2}$.
		\item[$\dagger_5$:] $f(s) = \overline{z}_{\be{1}} \be{1} + z_{\be{2}} \be{2}$.
		\item[$\ddagger_6$:] $f(s) = \overline{z}_{\be{2}} \be{1} + z_{\be{1}} \be{2}$.
		\item[$\ddagger_7$:] $f(s) = z_{\be{2}} \be{1} + \overline{z}_{\be{1}} \be{2}$.
		\end{enumerate}
		\end{multicols}
	Furthermore, we can extract some subgroups from the whole group. Here are the possible subgroups. Straightforward calculations show that
		\begin{itemize}
		\item $\{ \dagger_0 , \dagger_1 , \dagger_2 , \dagger_3 \}$, $\{ \dagger_0 , \dagger_3 , \dagger_4 , \dagger_5 \}$, and $\{ \dagger_0 , \dagger_3 , \ddagger_6 , \ddagger_7 \}$ are subgroups of order $4$. 
		\item $\{ \dagger_0 , \dagger_i \}$ for $i = 1, 2, \ldots , 5$ are subgroups of order $2$.
		\item $\{ \dagger_0 \}$ is a subgroup of order $1$.
		\end{itemize}
	These are the only subgroups of the group ${\boldsymbol \dagger}$. This is a consequence of the following theorem. We let $D_8$ be the dihedral group of the symmetries of the square. It is generated by two elements: a rotation of $90$ degrees counterclockwise, denoted by $a$ and a reflection about one of the lines joining midpoints of opposite sides, denoted by $x$. The operation is the composition of functions. More explicitly, we have
		\begin{align*}
		D_8 = \{ \mathrm{Id} , a, a^2, a^3 , x, ax, a^2 x, a^3 x \}
		\end{align*}
	where $\mathrm{Id}$ is the identity mapping (doing nothing on the square). The Cayley table is presented in Table \ref{Tb:CayleyTableD8} (see the Appendix).
		
	\begin{theorem}
	The group of all conjugates and pseudoconjugates is isomorphic (in the sense of group) to the dihedral group $D_8$, the group of symmetries of the square.
	\end{theorem}
	\begin{proof}
	In \cite[Table 1]{Newman1990}, the number of groups with $8$ elements is $5$ and this result is attributed to Cayley (1859). Since $\dagger_1 \circ \ddagger_6 \neq \ddagger_6 \circ \dagger_1$, the group ${\boldsymbol \dagger}$ is noncommutative. Amongs the $5$ groups with $8$ elements, there are two which are noncommutative: the quaternion group $\{ 1, -1, \im{} , -\im{} , \jm{} , -\jm{} , \mathbf{k} , -\mathbf{k} \}$ with the rules of multiplication of the units given in the introduction or the dihedral group $D_8$. Since $\pm\im{}$, $\pm\jm{}$, $\pm \mathbf{k}$ are elements of order $4$ (or, with our terminology, $4$-involutions) and the group ${\boldsymbol \dagger}$ has only $2$ elements of order $4$ (equivalently, $4$-involutions), we must conclude that ${\boldsymbol \dagger}$ is isomorphic, as a group, to $D_8$.
	\end{proof}
	Define the map $\rho : {\boldsymbol \dagger} \ra D_8$ by
		\begin{multicols}{4}
		\begin{itemize}
		\item $\rho (\dagger_0) := \mathrm{Id}$.
		\item $\rho (\dagger_1) := x$.
		\item $\rho (\dagger_2) := a^2 x$.
		\item $\rho (\dagger_3) := a^2$.
		\item $\rho (\dagger_4 ) := a^3 x$.
		\item $\rho (\dagger_5 ) := ax$.
		\item $\rho (\ddagger_6) := a$.
		\item $\rho (\ddagger_7 ) := a^3$.
		\end{itemize}
		\end{multicols}
	Comparing the Cayley table of the group ${\boldsymbol \dagger}$ (see Table \ref{Tb:CayleyEigthFunctions}) and the group $D_8$ (see Table \ref{Tb:CayleyTableD8}), we see that $\rho$ is a group isomorphism. All the proper subgroups of $D_8$ are known. There are exactly nine of them: 
		\begin{itemize}
		\item $\{ \mathrm{Id}, x, a^2, a^2x\}$, $\{\mathrm{Id}, ax, a^2, a^3 x\}$, $\{ \mathrm{Id}, a, a^2, a^3 \}$ (subgroups with four elements).
		\item $\{ \mathrm{Id}, x \}$, $\{ \mathrm{Id}, a^2 x \}$, $\{ \mathrm{Id}, a^2\}$, $\{ \mathrm{Id} , a^3 x \}$, $\{ \mathrm{Id}, ax \}$ (subgroups with two elements).
		\item $\{ \mathrm{Id} \}$ (trivial subgroup).
		\end{itemize}
		Applying $\rho^{-1}$ to the above subgroups, we obtain the list of subgroups of ${\boldsymbol \dagger }$ enumerated above.
		
	The group structure of the conjugates and pseudoconjugates gives rise to surprising facts on the invertibility of a bicomplex number. If we multiply a bicomplex number with all of its possible conjugates and pseudoconjugates, we obtain
		\begin{align*}
		s^{\dagger_0} s^{\dagger_1} s^{\dagger_2} s^{\dagger_3} s^{\dagger_4} s^{\dagger_5} s^{\ddagger_6} s^{\ddagger_7} = |z_{\be{1}}|^4 |z_{\be{2}}|^4 \be{1} + |z_{\be{1}}|^4 |z_{\be{2}}|^4 \be{2} = |z_{\be{1}}|^4 |z_{\be{2}}|^4 .
		\end{align*}
	This is a real number. If we instead multiply all the conjugates from the subgroup $\{ \dagger_0 , \dagger_1 , \dagger_2 , \dagger_3 \}$ together, we obtain
		\begin{align*}
		s^{\dagger_0} s^{\dagger_1} s^{\dagger_2} s^{\dagger_3} = |z_{\be{1}}|^2 |z_{\be{2}}|^2 \be{1} + |z_{\be{1}}|^2 |z_{\be{2}}|^2 \be{2} = |z_{\be{1}}|^2 |z_{\be{2}}|^2 .
		\end{align*}
	Again, we obtain a real number.	Similarly, if we multiply the conjugates in the subgroups $\{ \dagger_0 , \dagger_3, \dagger_4 , \dagger_5 \}$ or $\{ \dagger_0 , \dagger_3 , \ddagger_6 , \ddagger_7 \}$ together, we obtain
		\begin{align*}
		s^{\dagger_0} s^{\dagger_3} s^{\dagger_4} s^{\dagger_5} = s^{\dagger_0} s^{\dagger_3} s^{\ddagger_6} s^{\ddagger_7} = |z_{\be{1}}|^2 |z_{\be{2}}|^2 .
		\end{align*}
	Again we obtain a real number. This leads to the following characterization of the invertible elements in $\mB \mC$.
		\begin{theorem}
		A number $s \in \mB \mC$ is invertible if and only if one of the following conditions holds:
			\begin{itemize}
			\item $s s^{\dagger_1} s^{\dagger_2} s^{\dagger_3} \neq 0$.
			\item $s s^{\dagger_3} s^{\dagger_4} s^{\dagger_5} \neq 0$.
			\item $s s^{\dagger_3} s^{\ddagger_6} s^{\ddagger_7} \neq 0$.
			\item $s s^{\dagger_1} s^{\dagger_2} s^{\dagger_3} s^{\dagger_4} s^{\dagger_5} s^{\ddagger_6} s^{\ddagger_7} \neq 0$.
			\end{itemize}
		If one of the above condition is satisfied, then the four conditions hold and the inverse is given by
			\begin{align*}
			s^{-1} = \frac{s^{\dagger_1} s^{\dagger_2} s^{\dagger_3}}{s s^{\dagger_1} s^{\dagger_2} s^{\dagger_3}} = \frac{s^{\dagger_3} s^{\dagger_4} s^{\dagger_5}}{s s^{\dagger_3} s^{\dagger_4} s^{\dagger_5}} = \frac{s^{\dagger_3} s^{\ddagger_6} s^{\ddagger_7}}{s s^{\dagger_3} s^{\ddagger_6} s^{\ddagger_7}} = \frac{s^{\dagger_1} s^{\dagger_2} s^{\dagger_3} s^{\dagger_4} s^{\dagger_5} s^{\ddagger_6} s^{\ddagger_7}}{s s^{\dagger_1} s^{\dagger_2} s^{\dagger_3} s^{\dagger_4} s^{\dagger_5} s^{\ddagger_6} s^{\ddagger_7}} .
			\end{align*}
		\end{theorem}
		\begin{proof}
		If $s$ is invertible, then $z_{\be{1}}$ and $z_{\be{2}}$ are non zero complex numbers by Theorem \ref{T:InvertibilityBC}. Therefore, $|z_{\be{1}}| \neq 0$ and $|z_{\be{2}}| \neq 0$. This implies that the desire quantities are non zero. If one of the three quantities are not zero, then this implies that $|z_{\be{1}}| |z_{\be{2}}| \neq 0$. We therefore have $|z_{\be{1}}| \neq 0$ and $|z_{\be{2}} | \neq 0$ and so $s$ is invertible.
		
		If one of the conditions is satisfied, then $|z_{\be{1}}| |z_{\be{2}}| \neq 0$ and therefore all the other expressions are also non zero. By multiplying $s$ by the expression proposed for $s^{-1}$, we obtain
			\begin{align*}
			s s^{-1} = \frac{|z_{\be{1}}|^m |z_{\be{2}}|^m}{|z_{\be{1}}|^m |z_{\be{2}}|^m} = 1
			\end{align*}
		where $m = 2$ or $m = 4$. This completes the proof.
		\end{proof}
	This last result gives interesting expressions of the inverse of a bicomplex number in terms of the product of its conjugates and/or pseudoconjugates. These expressions are similar to the usual expression of the inverse of a complex number in terms of its complex conjugate. Also, the last identities were discovered by Vallière and Rochon \cite{RochonValiere} for the tricomplex numbers in a very specific situation. What is suprising in our last result is that the same property can be extended to the other groups of conjugates and pseudoconjugates.
	
\section{Geometry of Involutions}\label{sec:GeomInterpretation}
We end this paper by stating geometric interpretations for each conjugate and pseudoconjugate. We seek a geometric description similar to the geometric meaning the complex conjugate has. To do so, we will use the representations of the conjugates and pseudoconjugates in terms of four real coefficients. We will now think of a bicomplex number $s = x_1 + x_{\im{1}} \im{1} + x_{\im{2}} \im{2} + x_{\jm{1}} \jm{1}$ as a quadruplet $\mathbf{s} = (x_{1} , x_{\im{1}} , x_{\im{2}} , x_{\jm{1}}) \in \mR^4$.

We introduce three reflections of a vector $\mathbf{s}$:
	\begin{itemize}
	\item $R_{\im{1}} (\mathbf{s}) = (x_1 , -x_{\im{1}} , x_{\im{2}} , x_{\jm{1}})$ (reflection about the plane $x_{\im{1}} = 0$).
	\item $R_{\im{2}} (\mathbf{s}) = (x_1 , x_{\im{1}} , -x_{\im{2}} , x_{\jm{1}})$ (reflection about the plane $x_{\im{2}} = 0$).
	\item $R_{\jm{1}} (\mathbf{s}) = (x_1 , x_{\im{1}} , x_{\im{2}} , -x_{\jm{1}})$ (reflection about the plane $x_{\jm{1}} = 0$).
	\end{itemize}
The vector form of the expression for the first, second, and third conjugates become
	\begin{itemize}
	\item[$\dagger_1$:] $(x_{1} , x_{\im{1}} , x_{\im{2}} , x_{\jm{1}}) \mapsto (x_1 , x_{\im{1}} , - x_{\im{2}} , -x_{\jm{1}} )$;
	\item[$\dagger_2$:] $(x_1 , x_{\im{1}} , x_{\im{2}} , x_{\jm{1}} ) \mapsto (x_1 , -x_{\im{1}} , x_{\im{2}} , -x_{\jm{1}} )$;
	\item[$\dagger_3$:]  $(x_1 , x_{\im{1}} , x_{\im{2}} , x_{\jm{1}}) \mapsto (x_1 , -x_{\im{1}} , -x_{\im{2}} , x_{\jm{1}} )$.
	\end{itemize}
We immediately see the first conjugate $\dagger_1$ is equal to the composition $R_{\im{2}} \circ R_{\jm{1}}$. Therefore, the conjugate $\dagger_1$ reflects a bicomplex number about the plane $x_{\jm{1}} = 0$ and then reflects it about the plane $x_{\im{2}} = 0$. We also see that the second conjugate $\dagger_2$ is equal to the composition $R_{\im{1}} \circ R_{\jm{1}}$. Therefore the conjugate $\dagger_2$ reflects a bicomplex number about the plane $x_{\jm{1}} = 0$ and then reflects it about the plane $x_{\im{1}} = 0$. Similarly, the third conjugate is equal to the composition $R_{\im{1}} \circ R_{\im{2}}$. Therefore the conjugate $\dagger_3$ reflects a bicomplex number about the plane $x_{\im{2}}$ and then reflects it about the plane $x_{\im{1}} = 0$.

The two remaining conjugates are different in nature than the first four because they interchange two imaginary units. We will start with the fourth conjugate. Let $\mathbf{a}$ be a vector perpenticular to a plane $a x_1 + b x_{\im{1}} + c x_{\im{2}} + d x_{\jm{1}} = 0$. For example, the vector $\mathbf{a}$ can be choosen to be $(a, b, c, d)$. Then the reflection about this plane is the linear transformation given by (see, for example, \cite[p.363]{jacobson1995})
	\begin{align*}
	R_{\mathbf{a}} (\mathbf{s}) := \mathbf{s} - \frac{\mathbf{s} \cdot \mathbf{a}}{\mathbf{a} \cdot \mathbf{a}} \mathbf{a} .
	\end{align*}
If $\mathbf{a_4} = (0, 1, 1, 0)$, which defines the plane $x_{\im{1}} + x_{\im{2}} = 0$, then a simple computation gives
	\begin{align*}
	R_{\mathbf{a_4}} (\mathbf{s}) = (x_1 , -x_{\im{2}} , -x_{\im{1}} , x_{\jm{1}})
	\end{align*}
which is exactly the expression of $\dagger_4$ in vector form. Therefore, the conjugate $\dagger_4$ reflects a bicomplex number about the plane $x_{\im{1}} + x_{\im{2}} = 0$. On the other hand, if we define $\mathbf{a_5} = (0, 1, -1, 0)$, which defines the plane $x_{\im{1}} - x_{\im{2}} = 0$, then a simple computation gives
	\begin{align*}
	R_{\mathbf{a_5}} (\mathbf{s}) = (x_1 , x_{\im{2}} , x_{\im{1}} , x_{\jm{1}})
	\end{align*}
which is exactly the expression of $\dagger_5$ in vector form. Therefore the conjugate $\dagger_5$ reflects a bicomplex number about the plane $x_{\im{1}} - x_{\im{2}} = 0$.

Finally, the vector form of the pseudoconjugate $\ddagger_6$ is
	\begin{align*}
	\mathbf{s} \mapsto (x_1 , -x_{\im{2}} , x_{\im{1}} , -x_{\jm{1}}) .
	\end{align*}
We see that $\ddagger_6$ is the composition $R_{\im{1}} \circ R_{\jm{1}} \circ R_{\mathbf{a_5}}$ or the composition $R_{\im{2}} \circ R_{\jm{1}} \circ R_{\mathbf{a_4}}$. Therefore, there are at least two geometric interpretations of the pseudoconjugate $\ddagger_6$:
	\begin{itemize}
	\item it reflects a bicomplex number about the plane $x_{\im{1}} - x_{\im{2}} = 0$, then reflects it about the plane $x_{\jm{1}} = 0$, and then reflects it about the plane $x_{\im{1}} =0$;
	\item it reflects a bicomplex number about the plane $x_{\im{1}} + x_{\im{2}} = 0$, then reflects it about the plane $x_{\jm{1}} = 0$, and then reflects it about the plane $x_{\im{2}} = 0$.
	\end{itemize}
In a similar way, the pseudoconjugate $\ddagger_7$ is the composition $R_{\im{1}} \circ R_{\jm{1}} \circ R_{\mathbf{a_4}}$ or the composition $R_{\im{2}} \circ R_{\jm{1}} \circ R_{\mathbf{a_5}}$. Therefore, there are also at least two geometric interpretations of the pseudoconjugate $\ddagger_7$:
	\begin{itemize}
	\item it reflects a bicomplex number about the plane $x_{\im{1}} + x_{\im{2}} = 0$, then reflects it about the plane $x_{\jm{1}} = 0$, and then reflects it about the plane $x_{\im{1}} = 0$.
	\item it reflects a bicomplex number about the plane $x_{\im{1}} - x_{\im{2}} = 0$, then reflects it about the plane $x_{\jm{1}}$, and then reflects it about the plane $x_{\im{2}} = 0$.
	\end{itemize}
	
\section*{Acknowledgements}
The author would like to thank Dominic Rochon for fruitful discussions on the problem. The author would also like to thank William Verrault for reading a first draft of this manuscript and pointing out some valuable suggestions.

\section*{Funding}
The author's research is supported by a CRSNG post-doctoral scholarship and is partially supported by an FRQNT postdoctoral scolarship.

\bibliographystyle{plain}
\bibliography{Biblio.bib}

\appendix

\newpage

\section{List of Tables}

\begin{table}[ht]
	\begin{tabular}{|c||c|c|c|c|c|}
	\hline
	$\circ$ & $\dagger_0$ & $\dagger_1$ & $\dagger_2$ & $\dagger_3$ \\\hline\hline
	$\dagger_0$ & $\dagger_0$ & $\dagger_1$ & $\dagger_2$ & $\dagger_3$ \\
	$\dagger_1$ & $\dagger_1$ & $\dagger_0$ & $\dagger_3$ & $\dagger_2$ \\ 
	$\dagger_2$ & $\dagger_2$ & $\dagger_3$ & $\dagger_0$ & $\dagger_1$ \\
	$\dagger_3$ & $\dagger_3$ & $\dagger_2$ & $\dagger_1$ & $\dagger_0$ \\\hline
	\end{tabular}
	\vspace*{8pt}
	\caption{Cayley Table of the first four conjugates}\label{Tb:CayleyTabFirstFourConj}
	\end{table}
	
	\newpage
	
	\begin{table}[ht]
		\begin{tabular}{|c||c|c|c|c|c|c|c|c|}
		\hline
	$\circ$ & $\dagger_0$ & $\dagger_1$ & $\dagger_2$ & $\dagger_3$ & $\dagger_4$ & $\dagger_5$ & $\ddagger_6$ & $\ddagger_7$ \\\hline\hline
	$\dagger_0$ & $\dagger_0$ & $\dagger_1$ & $\dagger_2$ & $\dagger_3$ & $\dagger_4$ & $\dagger_5$ & $\ddagger_6$ & $\ddagger_7$ \\
	$\dagger_1$ & $\dagger_1$ & $\dagger_0$ & $\dagger_3$ & $\dagger_2$ & $\ddagger_6$ & $\ddagger_7$ & $\dagger_4$ & $\dagger_5$ \\ 
	$\dagger_2$ & $\dagger_2$ & $\dagger_3$ & $\dagger_0$ & $\dagger_1$ & $\ddagger_7$ & $\ddagger_6$ & $\dagger_5$ & $\dagger_4$\\
	$\dagger_3$ & $\dagger_3$ & $\dagger_2$ & $\dagger_1$ & $\dagger_0$ & $\dagger_5$ & $\dagger_4$ & $\ddagger_7$ & $\ddagger_6$ \\
	$\dagger_4$ & $\dagger_4$ & $\ddagger_7$ & $\ddagger_6$ & $\dagger_5$ & $\dagger_0$ & $\dagger_3$ & $\dagger_2$ & $\dagger_1$ \\
	$\dagger_5$ & $\dagger_5$ & $\ddagger_6$ & $\ddagger_7$ & $\dagger_4$ & $\dagger_3$ & $\dagger_0$ & $\dagger_1$ & $\dagger_2$ \\
	$\ddagger_6$ & $\ddagger_6$ & $\dagger_5$ & $\dagger_4$ & $\ddagger_7$ & $\dagger_1$ & $\dagger_2$ & $\dagger_3$ & $\dagger_0$ \\
	$\ddagger_7$ & $\ddagger_7$ & $\dagger_4$ & $\dagger_5$ & $\ddagger_6$ & $\dagger_2$ & $\dagger_1$ & $\dagger_0$ & $\dagger_3$ \\\hline
		\end{tabular}
		\vspace*{8pt}
		\caption{Cayley table for the six conjugates and the two pseudo-conjugates}\label{Tb:CayleyEigthFunctions}
		\end{table}
		
	\newpage
	
		\begin{table}[ht]
		\begin{tabular}{|c||c|c|c|c|c|c|c|c|}
		 \hline $\circ$ & $\mathrm{Id}$ & $x$ & $a^2x$ & $a^2$ & $a^3 x$ & $ax$ & $a$ & $a^3$ \\\hline\hline
		$\mathrm{Id}$ & $\mathrm{Id}$ & $x$ & $a^2x$ & $a^2$ & $a^3 x$ & $ax$ & $a$ & $a^3$ \\\hline
	$x$ & $x$ & $\mathrm{Id}$ & $a^2$ & $a^2 x$ & $a$ & $a^3$ & $a^3 x$ & $ax$ \\\hline
		$a^2x$ & $a^2 x$ & $a^2$ & $\mathrm{Id}$ & $x$ & $a^3$ & $a$ & $ax$ & $a^3x$ \\\hline
		$a^2$ & $a^2$ & $a^2 x$ & $x$ & $\mathrm{Id}$ & $ax$ & $a^3 x$ & $a^3$ & $a$ \\\hline
		$a^3 x$ & $a^3 x$ & $a^3$ & $a$ & $ax$ & $\mathrm{Id}$ & $a^2$ & $a^2 x$ & $x$ \\\hline
		$ax$ & $ax$ & $a$ & $a^3$ & $a^3x$ & $a^2$ & $\mathrm{Id}$ & $x$ & $a^2x$ \\\hline
		$a$ & $a$ & $ax$ & $a^3 x$ & $a^3$ & $x$ & $a^2x$ & $a^2$ & $\mathrm{Id}$ \\\hline
		$a^3$ & $a^3$ & $a^3 x$ & $ax$ & $a$ & $a^2 x$ & $x$ & $\mathrm{Id}$ & $a^2$ \\\hline
		\end{tabular}
		\vspace*{8pt}
		\caption{Cayley table for $D_8$}\label{Tb:CayleyTableD8}
		\end{table}

\end{document}